\documentclass[a4paper,11pt]{amsart}

\usepackage{graphicx}
\usepackage{amsmath}
\usepackage{amssymb}
\usepackage{hyperref}

\newtheorem{thm}{Theorem}
\newtheorem{lem}{Lemma}%
\newtheorem{cor}{Corollary}%

\newcounter{fig}

\theoremstyle{definition}

\theoremstyle{remark}
\newtheorem{remark}{Remark}

\theoremstyle{plain}

\numberwithin{equation}{section}

\def\CC{{\mathbb C}}

\def\NN{{\mathbb N}}
\def\QQ{{\mathbb Q}}

\def\RR{{\mathbb R}}
\def\TT{{\mathbb T}}

\def\ZZ{{\mathbb Z}}

\def\scrB{{\mathcal B}}

\def\scrD{{\mathcal D}}

\def\Re{\operatorname{Re}}
\def\Im{\operatorname{Im}}

\def\e{\mathrm{e}}
\def\i{\mathrm{i}}

\def\C{\operatorname{C{}}}

\def\L{\operatorname{L{}}}

\def\sgn{\operatorname{sgn}}

\title{The value distribution of incomplete Gauss sums}
\author{Emek Demirci Akarsu}\address{School of Mathematics, University of Bristol,
Bristol BS8 1TW, U.K.\newline
\rule[0ex]{0ex}{0ex} \hspace{8pt}{\tt E.DemirciAkarsu@bristol.ac.uk}}
\author{Jens Marklof}
\address{School of Mathematics, University of Bristol,
Bristol BS8 1TW, U.K.\newline
\rule[0ex]{0ex}{0ex} \hspace{8pt}{\tt j.marklof@bristol.ac.uk}}
\date{6 July 2012}
\thanks{E.D.A.\ is supported by a Turkish Ministry of Education doctoral training grant. J.M.\ is supported by a Royal Society Wolfson Research Merit Award, a Leverhulme Trust Research Fellowship and ERC Advanced Grant HFAKT}

\begin{document}

\begin{abstract}
It is well known that the classical Gauss sum, normalized by the square-root number of terms, takes only finitely many values. If one restricts the range of summation to a subinterval, a much richer structure emerges. We prove a limit law for the value distribution of such incomplete Gauss sums. The limit distribution is given by the distribution of a certain family of periodic functions. Our results complement Oskolkov's pointwise bounds for incomplete Gauss sums as well as the limit theorems for quadratic Weyl sums (theta sums) due to Jurkat and van Horne and the second author.
\end{abstract}

\maketitle
\section{Introduction \label{secIntro}}
The present paper investigates the asymptotic distribution of the incomplete Gauss sum
\begin{equation}
g_\varphi(p,q)=\sum_{h=0}^{q-1} \varphi\bigg(\frac{h}{q}\bigg) e_q(p h^2),
\end{equation}
where $q\in\NN$, $p\in\ZZ_q$, and $e_q(x)=\e^{2\pi\i x/q}$; the weight function $\varphi$ is periodic with period one. The case $\varphi=1$ corresponds to the classical Gauss sum. The main example of an incomplete Gauss sum in the literature is the case when $\varphi$ is the characteristic function of a subinterval of the unit interval \cite{Lehmer76,Fiedler77,Oskolkov91,Evans03,Paris05,Paris08}.

It is natural to assume that $p$ and $q$ are coprime, i.e, $p\in\ZZ_q^\times = \{p\leq q, \gcd(p,q)=1 \}$. Here $\ZZ_q^\times$ is the multiplicative group of integers mod $q$. The order of  $\ZZ_q^\times$ is denoted by $\phi(q)$ (Euler's totient function).  If $p,q$ are not coprime, say $\gcd(p,q)=r$ for some $r>1$, we set $p'=p/r$, $q'=q/r$, and observe that
\begin{equation}\label{reduct}
g_\varphi(p,q)= g_{\varphi_r}(p',q')
\end{equation}
where 
\begin{equation}
\varphi_r(x)=\sum_{k=0}^{r-1} \varphi\bigg( \frac{x+k}{r} \bigg). 
\end{equation}
The case when $p,q$ are not coprime can therefore be reduced to the coprime case.

Functional equations and pointwise estimates of incomplete Gauss sums have been studied extensively by Oskolkov \cite{Oskolkov91}, and the aim of the present paper is to complement his results by establishing limit theorems for their value distribution at random argument. 

The existence of a limit distribution of the classical theta sum
\begin{equation}
S_N(x) = \frac{1}{\sqrt{N}} \sum_{n=0}^{N-1} e(n^2 x), \qquad e(x)=\e^{2\pi\i x},
\end{equation}
for $x$ uniformly distributed in $\TT=\RR/\ZZ$ has been proved by Jurkat and van Horne \cite{Jurkat81,Jurkat82,Jurkat83} (for its absolute value) and the second author \cite{Marklof99} (for its full distribution in the complex plane); we refer the reader also to the recent study by Cellarosi \cite{Cellarosi11}. A striking feature of the limit distribution of theta sums is that it has a heavy tail: The probability that $|S_N(x)|$ has a value greater than $R$, decays, for large $R$, as $R^{-4}$. At rational $x$, the theta sum of course reduces to an incomplete Gauss sum where $\varphi$ is the characteristic function of an interval, and we will see below (Remark \ref{rem1}) that in this case the limit distribution has compact support---the exact opposite of a heavy tail. 

We denote by $\sum_{k\in\ZZ} \widehat\varphi_k e(k x)$ the Fourier series of $\varphi$. 
We will focus for the major part of this paper on Gauss sums with differentiable weight functions $\varphi$ in the space
\begin{equation}
\scrB(\TT) = \bigg\{ \varphi :  \sum_{k\in\ZZ} k^2 |\widehat\varphi_k| <\infty\bigg\} ,
\end{equation}
and only later extend our results to general Riemann integrable functions, under an additional assumption on $q$.

The limit distributions of incomplete Gauss sums will be characterized by the following random variables: 
\begin{itemize}
\item $X$ takes the four values $\pm 1\pm\i$ with equal probability. 
\item $Y$ takes the values $\pm 1$ with equal probability. 
\item $Z$ takes the values $1 \pm \i$ with equal probability.
\item $G_\varphi^+$, $G_\varphi$, $G_\varphi^-$ are random variables given by the Fourier series
\begin{equation}\label{FS}
 G_\varphi^+(x)=\sum_{n\in\ZZ} \widehat \varphi_{2n} \; e(n^2 x), \qquad
 G_\varphi(x)=\sum_{n\in\ZZ} \widehat \varphi_{n} \; e(n^2 x), \qquad
 G_\varphi^-(x)=\sum_{n\in2\ZZ+1} \widehat \varphi_{n} \; e(n^2 x) ,
\end{equation}
respectively, with $x$ uniformly distributed on $\TT$. 
\end{itemize}

\begin{remark}
Note that, for $\varphi\in\scrB(\TT)$, the functions in \eqref{FS} are differentiable and thus continuous. If $\varphi$ satisfies the functional relation $\varphi(x)=\varphi(\frac{1}{2}-x)$, then its Fourier coefficients are related via $\hat\varphi_{-n}=(-1)^n\hat\varphi_n$, and hence $G_\varphi^-=0$ and  $G_\varphi^+=G_\varphi$. If $\widehat \varphi_{n}+\widehat \varphi_{-n}$ is real-valued for all $n$, then $\Im G_\varphi^\pm(-x)=-\Im G_\varphi^\pm(x)$ and $\Im G_\varphi(-x)=-\Im G_\varphi(x)$. Hence the probability density describing the distribution of the imaginary part of the random variables $G_\varphi^+$, $G_\varphi$, $G_\varphi^-$ is symmetric. Furthermore, $\Im G_\varphi^-(x+\frac14)=\Re G_\varphi^-(x)$, and thus the real and imaginary part of $G_\varphi^-$ have the same distribution. 
\end{remark}

We define $\epsilon_a = 1$ if $a\equiv 1 \bmod 4$, and $\epsilon_a = \i$ if $a\equiv 3 \bmod 4$. The symbol $\xrightarrow{d}$ denotes convergence in distribution. 

\begin{thm}\label{thm1}
Fix a subset $\scrD\subset\TT$ with boundary of measure zero, and let $\varphi\in\scrB(\TT)$. For each $q\in\NN$, choose $p\in\ZZ_q^\times\cap q\scrD$ at random with uniform probability. Then, as $q\to\infty$ along an appropriate subsequence as specified below, we have:

\begin{center}
\renewcommand{\arraystretch}{2}
\begin{tabular}{|c|c|c|}
\hline
& $q$ is not a square & $q$ is a square \\
\hline
$q\equiv 0\bmod 4$ & $\displaystyle\bigg( \frac{g_1(p,q)}{\sqrt{q}}, \frac{g_\varphi(p,q)}{g_1(p,q)} \bigg) \xrightarrow{d} (X,G_\varphi^+)$  & $\displaystyle\bigg( \frac{g_1(p,q)}{\sqrt{q}}, \frac{g_\varphi(p,q)}{g_1(p,q)} \bigg) \xrightarrow{d} (Z,G_\varphi^+)$ \\[10pt]
\hline
$q\equiv 1\bmod 2$ & $\displaystyle\bigg( \frac{g_1(p,q)}{\epsilon_q\sqrt q}, \frac{g_\varphi(p,q)}{g_1(p,q)} \bigg) \xrightarrow{d}  (Y, G_\varphi)$ & $\displaystyle \frac{g_\varphi(p,q)}{\epsilon_q\sqrt q}  \xrightarrow{d}  G_\varphi $ \\[10pt]
\hline
\hline
& $q/2$ is not a square & $q/2$ is a square \\
\hline
$q\equiv 2\bmod 4$ & $\displaystyle\bigg( \frac{g_1(2p,q/2)}{\epsilon_{q/2}\sqrt{q/2}}, \frac{g_\varphi(p,q)}{2g_1(2p,q/2)} \bigg) \xrightarrow{d}  (Y, G_\varphi^-)$  & $\displaystyle\frac{g_\varphi(p,q)}{\epsilon_{q/2} \sqrt{2q}}  \xrightarrow{d}  G_\varphi^-$ \\[10pt]
\hline
\end{tabular}
\end{center}
\end{thm}

The proof of this theorem is given in Section \ref{proof}. The key ingredients are functional equations for incomplete Gauss sums (Section \ref{func}) and estimates on (twisted) Kloosterman sums and Sali\'e sums (Section \ref{equi}). It is crucial that the exponential sums considered here are quadratic in $h$. The case of higher powers is significantly more difficult, cf.~\cite{Montgomery95}.  

To illustrate the statement of Theorem \ref{thm1}, let us consider the distribution of the absolute values of incomplete Gauss sums on the positive axis $\RR_{\geq 0}$. 

\begin{cor}\label{cor1}
Under the assumptions of Theorem \ref{thm1}, as $q\to\infty$,

\begin{center}
\renewcommand{\arraystretch}{2}
\begin{tabular}{|c|c|}
\hline
$q\equiv 0\bmod 4$ & $\displaystyle \frac{|g_\varphi(p,q)|}{\sqrt{2q}}  \xrightarrow{d} |G_\varphi^+|$  \\[10pt]
\hline
$q\equiv 1\bmod 2$ & $\displaystyle \frac{|g_\varphi(p,q)|}{\sqrt{q}}  \xrightarrow{d}  |G_\varphi|$  \\[10pt]
\hline
$q\equiv 2\bmod 4$ & $\displaystyle \frac{|g_\varphi(p,q)|}{\sqrt{2q}} \xrightarrow{d}  |G_\varphi^-|$   \\[10pt]
\hline
\end{tabular}
\end{center}
\end{cor}

Since for smooth $\varphi$ the normalized incomplete Gauss sums $q^{-1/2} |g_\varphi(p,q)|$ are bounded (cf.~Theorem \ref{FEthm} below), the previous corollary implies convergence of the $k$th moment
\begin{equation}
M_{k,\varphi}(q)=\frac{1}{\phi(q) |\scrD|} \sum_{p\in\ZZ_q^\times\cap q\scrD} |g_\varphi(p,q)|^k .
\end{equation}

\begin{cor}\label{cor2}
Under the assumptions of Theorem \ref{thm1}, we have for any $k\geq 0$

\begin{center}
\renewcommand{\arraystretch}{2}
\begin{tabular}{|c|c|}
\hline
$q\equiv 0\bmod 4$ & $\displaystyle \lim_{q\to\infty} \frac{M_{k,\varphi}(q)}{(2q)^{k/2}}  = \int_\TT |G_\varphi^+(x)|^k dx$  \\[10pt]
\hline
$q\equiv 1\bmod 2$ & $\displaystyle \lim_{q\to\infty} \frac{M_{k,\varphi}(q)}{q^{k/2}}  = \int_\TT |G_\varphi(x)|^k dx$  \\[10pt]
\hline
$q\equiv 2\bmod 4$ & $\displaystyle \lim_{q\to\infty} \frac{M_{k,\varphi}(q)}{(2q)^{k/2}}  = \int_\TT |G_\varphi^-(x)|^k dx $  \\[10pt]
\hline
\end{tabular}
\end{center}
\end{cor}

The following technical estimate allows us to extend Theorem \ref{thm1} to non-smooth $\varphi$ as long as $q$ has a bounded number of divisors $d(q)$.

\begin{lem}\label{RIlem}
Fix a positive integer $N$. Then there exists a constant $C_N>0$ such that, for  every Riemann integrable function $\varphi:\TT\to\CC$, we have
\begin{equation}\label{varia}
\limsup_{\substack{q\to\infty \\ d(q)\leq N}} \frac{M_{2,\varphi}(q)}{q} \leq \frac{C_N}{|\scrD|} \| \varphi \|_2^2.
\end{equation}
\end{lem} 

Lemma \ref{RIlem} is proved in Section \ref{variasec}. Together with Chebyshev's inequality it implies the following extension of Theorem \ref{thm1}.

\begin{thm}\label{thm2}
Fix a subset $\scrD\subset\TT$ with boundary of measure zero, and let $\varphi:\TT\to\CC$ be Riemann integrable. Then the conclusions of Theorem \ref{thm1} and Corollary \ref{cor1} remain valid for any sequence of $q\to\infty$ with a bounded number of divisors. 
\end{thm}

The proof of Theorem \ref{thm2} is supplied in Section \ref{extens}.

\begin{remark}\label{rem1}
For $\varphi\in\scrB(\TT)$, the functions in \eqref{FS} are continuous and bounded. Arkhipov and Oskolkov \cite{Arkhipov87,Oskolkov91} prove that boundedness (but not continuity) still holds even if $\varphi$ is the characteristic function of a subinterval of $\TT$. This implies that the limit distribution has compact support  and that therefore Corollary \ref{cor2} remains valid in this case, subject to the addtional assumption $d(q)\leq N$.
\end{remark}

\begin{remark}\label{rem2}
It seems plausible that the hypothesis on the number of divisors of $q$ can be removed from Theorem \ref{thm2}, at least in the case of weights $\varphi$ of bounded variation---or indeed all Riemann integrable functions. 
\end{remark}

To illustrate Theorem \ref{thm2}, we have computed numerically the value distribution of the real and imaginary parts of incomplete Gauss sums for different values of $q$, see Figures \ref{fig5012}--\ref{fig5014} and Section \ref{secNumerics}.

\begin{figure}
\begin{center}
\includegraphics[width=0.49\textwidth]{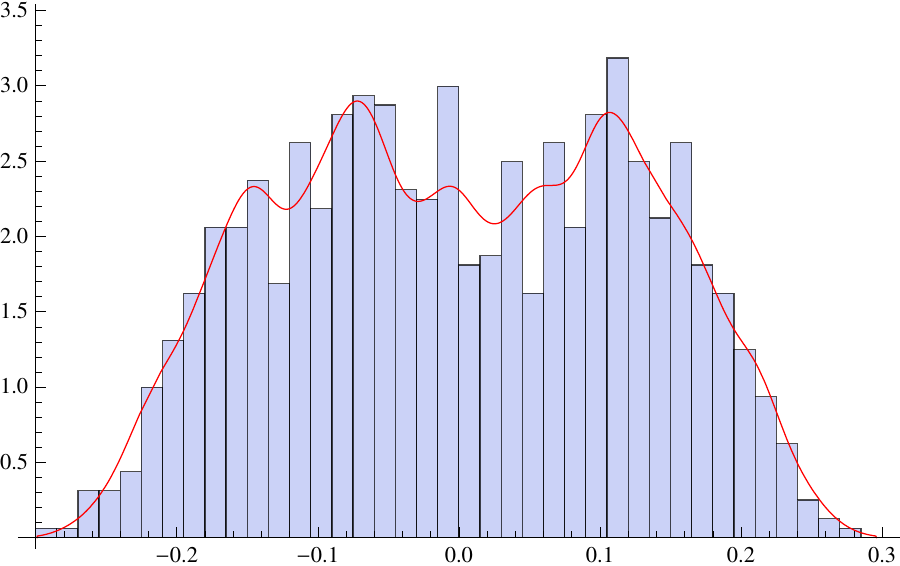}
\includegraphics[width=0.49\textwidth]{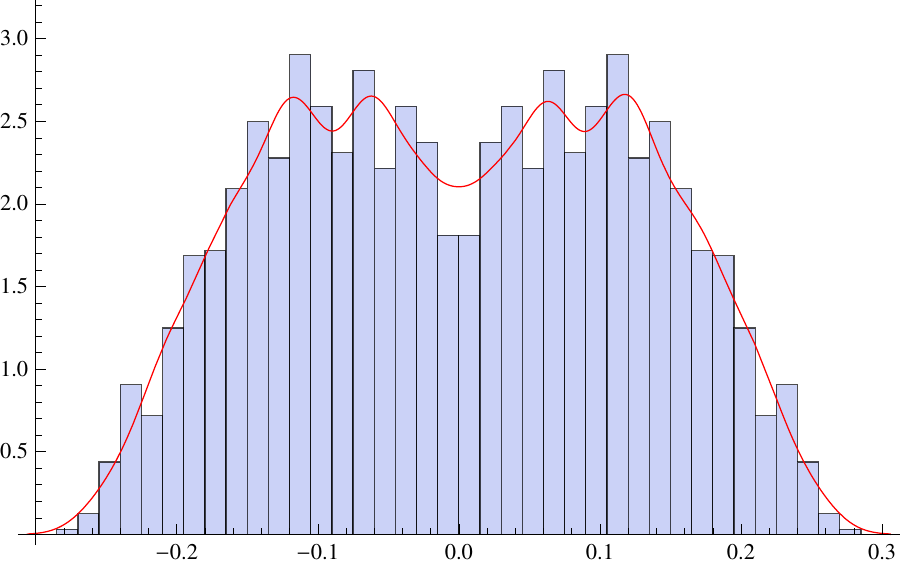}
\end{center}
\caption{The histogram on the left shows the value distribution of the real part of $\frac{g_\varphi(p,q)}{g_1(p,q)}-\frac{1}{\sqrt7}$ where $\varphi$ is the characteristic function of the interval $[0,\frac{1}{\sqrt7}]$, $q=5012=2^2\times 7\times 179$ and $p$ uniformly distributed in $\ZZ_q^\times$. There are thus $\phi(q)=2136$ sample points distributed across 40 bins, which means we have on average 53.4 values in each bin. The histogram seems consistent with fluctuations of the order of the square-root of that number, i.e., approximately $14\%$ of the height of each bin. The histogram on the right shows the imaginary part of the corresponding quantities. The continuous curves represent a numerical approximation to the real and imaginary part of the probability density of the random variable $G_\varphi^+(x)$.} \label{fig5012}
\end{figure}
\begin{figure}
\begin{center}
\includegraphics[width=0.49\textwidth]{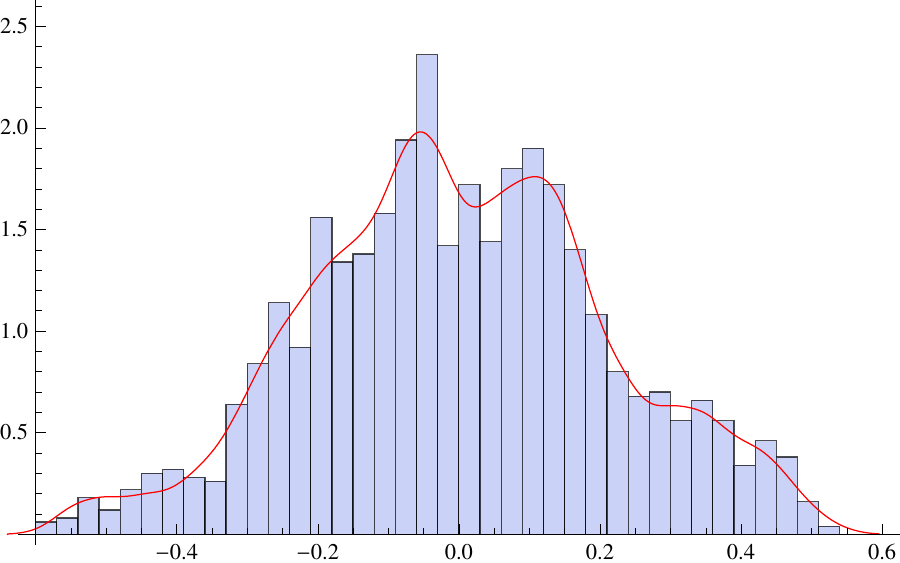}
\includegraphics[width=0.49\textwidth]{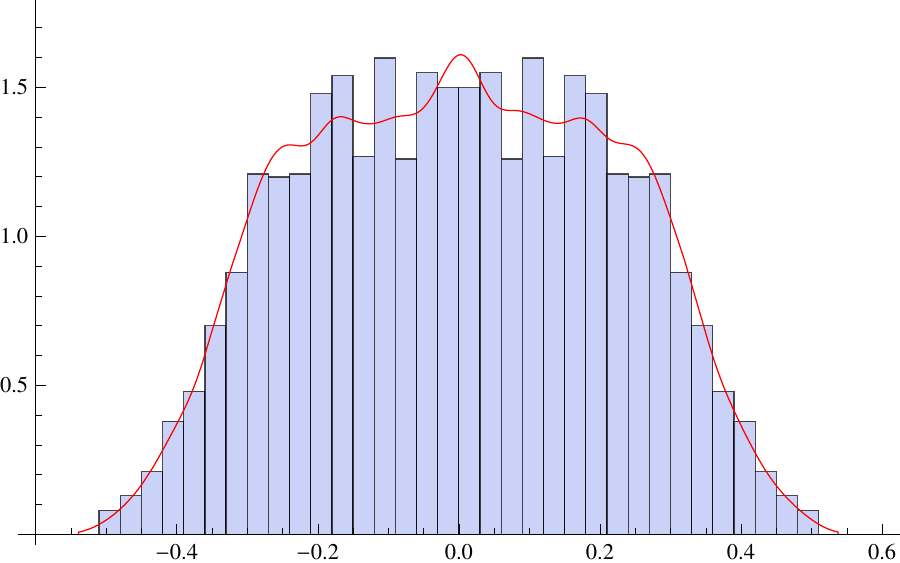}
\end{center}
\caption{Same as Figure \ref{fig5012}, now for $q=5013=3^2\times 557$. The number of sample points is now $\phi(q)=3336$. The continuous curves represent a numerical approximation to the real and imaginary part of the probability density of the random variable $G_\varphi(x)$.} \label{fig5013}
\end{figure}
\begin{figure}
\begin{center}
\includegraphics[width=0.49\textwidth]{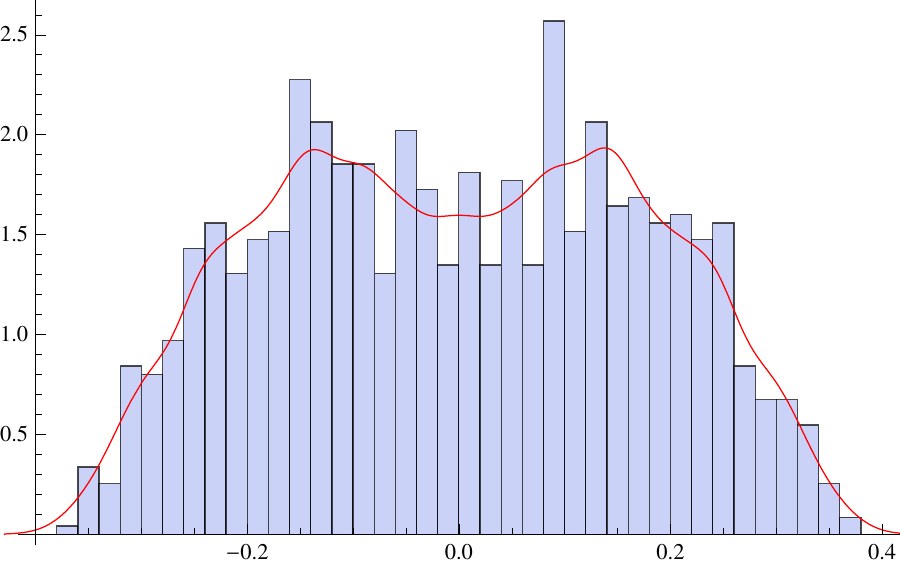}
\includegraphics[width=0.49\textwidth]{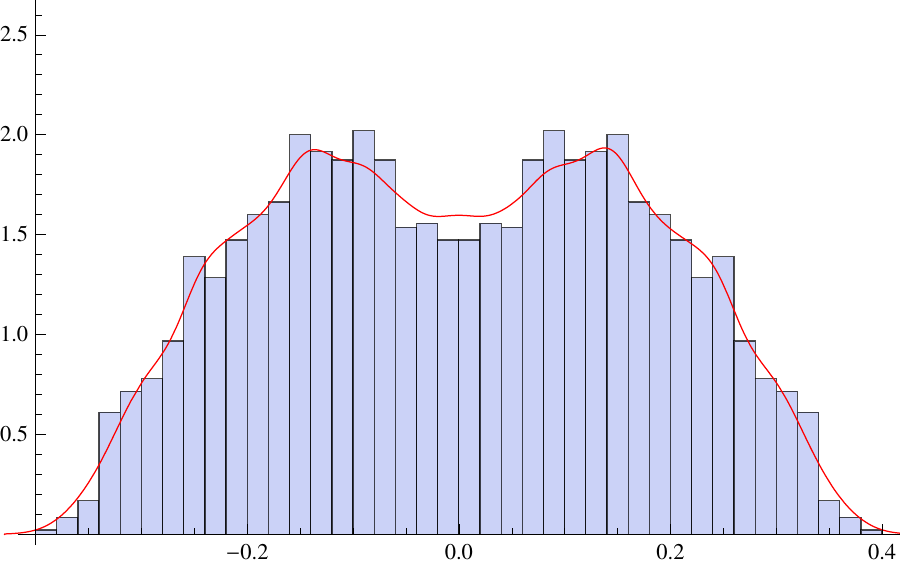}
\end{center}
\caption{The histogram on the left shows the value distribution of the real part of $\frac{g_\varphi(p,q)}{2g_1(2p,q/2)}$ with the same $\varphi$ as in Figure \ref{fig5012} and $q=5014=2\times23\times109$, $p$ uniformly distributed in $\ZZ_q^\times$, where $\phi(q)=2376$. The histogram on the right shows the imaginary part of the corresponding quantities. The continuous curves represent a numerical approximation to the real and imaginary part of the probability density of the random variable $G_\varphi^-(x)$.} \label{fig5014}
\end{figure}

\section{Functional equations for incomplete Gauss sums\label{func}} 

Legendre's quadratic residue symbol is defined for an odd prime $b$ by
\begin{equation} 
\left(\frac{a}{b}\right)=
\begin{cases}               
+1 & \text{if $b\nmid a$ and $a$ is a quadratic residue}\\                   
0 & \text{if $b\mid a$}\\
-1& \text{if $b\nmid a$ and $a$ is a quadratic non residue.}            
\end{cases}  
\end{equation}
Following Jacobi, we extend the definition to arbitrary odd integers $b$ multiplicatively:
Let $b$ a positive odd integer with prime factorization $\prod_{i=1}^s p_i^{r_i}$. For $a\in\ZZ$ we define $(\frac{a}{1})=1$ and $(\frac{a}{b})=\prod_{i=1}^s(\frac{a}{p_i})^{r_i}$. The generalized quadratic residue symbol (or Jacobi symbol) $\left(\frac{a}{b}\right)$ is characterized by the following properties (cf.~\cite{Shimura73}):
\begin{list}{\rm (\roman{fig})}{\usecounter{fig}}
\item $\left(\frac{a}{b}\right)=0$ if $\gcd(a,b)\neq 1$.
\item If $b$ is an odd prime, $\left(\frac{a}{b}\right)$ coincides
with the ordinary quadratic residue symbol.
\item If $b>0$, $\left(\frac{\cdot}{b}\right)$ defines a character
modulo $b$.
\item If $a\neq 0$, $\left(\frac{a}{\cdot}\right)$ defines a character
modulo a divisor of $4a$, whose conductor is the conductor of
$\QQ(\sqrt a)$ over $\QQ$. 
\item $\left(\frac{a}{-1}\right)=\sgn a$.
\item $\left(\frac{0}{\pm 1}\right)= 1$.
\end{list}
In particular $\left(\frac{a}{b}\right)^2$=1, if $\gcd(a,b)=1$.

We assume from now on that $\gcd(p,q)=1.$

The classical Gauss sum
\begin{equation}
g_1(p,q)=\sum_{h\bmod q} e_q(p h^2) .
\end{equation}
can be evaluated explicitly in terms of the Jacobi symbol:
\begin{equation}\label{GS} 
g_1(p,q)=
\begin{cases}
(1+\i)\; \epsilon_p^{-1} (\frac{q}{p})\; \sqrt q
& \text{if $q\equiv 0\bmod 4$} \\ 
\epsilon_q (\frac{p}{q}) \;\sqrt q
& \text{if $q\equiv 1\bmod 2$}\\
0 & \text{if $q\equiv 2\bmod 4$,} 
\end{cases}
\end{equation}
with $\epsilon_a$ as defined before Theorem \ref{thm1}.

The following theorem is implicit in the papers of Fiedler, Jurkat and K\"orner \cite{Fiedler77} and Oskolkov \cite{Oskolkov91}

\begin{thm}\label{FEthm}
For $\varphi\in\scrB(\TT)$,
\begin{equation} 
g_\varphi(p,q)=
\begin{cases}
g_1(p,q) \, G_\varphi^+\big(-\frac{\overline{p}}{q}\big)
& \text{if $q\equiv 0\bmod 4$} \\[5pt]
g_1(p,q) 
 \, G_\varphi\big(-\frac{\overline{4p}}{q}\big)
& \text{if $q\equiv 1\bmod 2$}\\[5pt]
2 g_1(2 p,q/2) 
 \, G_\varphi^-\big(-\frac{\overline{8p}}{q/2}\big)
& \text{if $q\equiv 2\bmod 4$.}
\end{cases}
\end{equation}
(In the first and second case, $\overline x$ denotes the inverse of $x\bmod q$, in the third the inverse $\bmod$ $q/2$.)
\end{thm}

\begin{proof}
Since the Fourier series of $\varphi$ is absolutely convergent, we may assume without loss of generality that $\varphi(x)=e(kx)$ with $k\in\ZZ$ fixed. The proof is a then simple exercise in completing the square (cf.\ \cite{Fiedler77}):

(i) $q\equiv 0\bmod 4$: For $k=2n$ even,
\begin{equation}
\sum_{h\bmod q}  e_q(p h^2 + 2n h)
= \sum_{h\bmod q}  e_q\big(p (h+n \overline p)^2-p (n \overline p)^2\big) = g_1(p,q)\;e_q(-\overline{p}\, n^2).
\end{equation}
For $k$ odd,
\begin{equation}
\begin{split}
\sum_{h\bmod q}  e_q(p h^2 + k h)
& = \sum_{h\bmod q}  e_q\bigg(p \bigg(h+\frac{q}{2}\bigg)^2 + k \bigg(h+\frac{q}{2}\bigg)\bigg) \\
& = - \sum_{h\bmod q}  e_q(p h^2 + k h),
\end{split}
\end{equation}
and therefore must be zero.

(ii) $q\equiv 1\bmod 2$: For $q$ odd, we may use the inverse of 2 mod $q$:
\begin{equation}\label{twi}
\sum_{h\bmod q}  e_q(p h^2 + k h)
= \sum_{h\bmod q}  e_q\big(p (h+\overline{2p}\,k)^2-p (\overline{2p}\,k)^2\big) = g_1(p,q)\;e_q(-\overline{4p}\, k^2).
\end{equation}

(iii) $q\equiv 2\bmod 4$: We deduce claim (iii) from the previous case: 
Define $q_0=q/2$ and $p_0=\frac14(2p-q)$. Clearly $q_0=1\bmod 2$, $\gcd(p_0,q_0)=1$ and $\frac{p}{q} = \frac{p_0}{q_0} + \frac12$. Then
\begin{equation}\label{rhs12}
\sum_{h\bmod q}  e_q(p h^2 + k h)
= \sum_{h\bmod 2q_0}  e_{q_0} \bigg(p_0 h^2 + \frac{q_0 h^2}{2} + \frac{kh}{2} \bigg)
= \sum_{h\bmod 2q_0}  e_{q_0} \bigg(p_0 h^2  + \frac{(k+q_0)h}{2} \bigg).
\end{equation}
Hence, for $k$ odd, $k+q_0$ is even and eq.~\eqref{twi} yields that the right hand side of \eqref{rhs12} equals
\begin{equation}
\begin{split}
2 g_1(p_0,q_0)e_{q_0}\bigg(-\overline{4p_0}\,  \bigg(\frac{k+q_0}{2}\bigg)^2\bigg)
& = 2 g_1(p_0,q_0)e_{q_0}\big(-\overline{16 p_0}\,  (k+q_0)^2\big) \\
& = 2 g_1(p_0,q_0)e_{q_0}\big(-\overline{16 p_0}\,  k^2\big) \\
& = 2 g_1(\overline{2} p,q/2)e_{q/2}\big(-\overline{8 p}\,  k^2\big) .
\end{split}
\end{equation}
Furthermore, $g_1(\overline{2} p,q/2) =g_1(2 p,q/2)$, since $(\frac{\overline{2}}{q/2})=(\frac{2}{q/2})$.
If $k$ is even,
\begin{equation}
\begin{split}
\sum_{h\bmod 2q_0}  e_{q_0} \bigg(p_0 h^2  + \frac{(k+q_0)h}{2} \bigg)
& =\sum_{h\bmod 2q_0}  e_{q_0} \bigg(p_0 (h+q_0)^2  + \frac{(k+q_0)(h+q_0)}{2} \bigg) \\
& =-\sum_{h\bmod 2q_0}  e_{q_0} \bigg(p_0 h^2  + \frac{(k+q_0)h}{2} \bigg).
\end{split}
\end{equation}
This term therefore vanishes.
\end{proof}

\section{Equidistribution mod $q$\label{equi}}

The functional equations of incomplete Gauss sums stated in Theorem \ref{FEthm} lead us to consider the joint distribution of $\frac{p}{q}$ and $\frac{\overline p}{q}$ on the torus $\TT^2$. The following statement is the second main ingredient in the proof of Theorem \ref{thm1}.

\begin{thm}\label{eqdthm} 
Let $f\in\C(\TT^2)$. Then the following convergence holds uniformly in $t\in\ZZ_q^\times$ as $q\to\infty$:
\begin{enumerate}
\item[(i)] For any sequence of $q$,
\begin{equation}\frac{1}{\phi(q)}\sum_{p\in\ZZ_q^\times}f\bigg(\frac{p}{q},\frac{t\overline p}{q}\bigg)\rightarrow \int_{\TT^2}f(x)dx .
\end{equation}
\item[(ii)]If $q\equiv 0\bmod 4$ is not a square then, for every $\sigma\in\{\pm 1,\pm\i\}$,
\begin{equation}\frac{1}{\phi(q)}\sum_{\substack{p\in\ZZ_q^\times \\ \epsilon_p (\frac{q}{p})=\sigma}}f\bigg(\frac{p}{q},\frac{t\overline p}{q}\bigg)\rightarrow\frac{1}{4}\int_{\TT^2}f(x)dx .
\end{equation}
\item[(iii)]If $q\equiv 0\bmod 4$ then, for every $\sigma\in\{\pm 1\}$,
\begin{equation}\frac{1}{\phi(q)}\sum_{\substack{p\in\ZZ_q^\times \\ p\equiv \sigma\bmod 4}}f\bigg(\frac{p}{q},\frac{t\overline p}{q}\bigg)\rightarrow\frac{1}{2}\int_{\TT^2}f(x)dx .
\end{equation}
\item[(iv)]If $q\equiv 1\bmod 2$ is not a square then, for every $\sigma\in\{\pm 1\}$,
\begin{equation}\frac{1}{\phi(q)}\sum_{\substack{p\in\ZZ_q^\times \\ (\frac{p}{q})=\sigma}}f\bigg(\frac{p}{q},\frac{t\overline p}{q}\bigg)\rightarrow\frac{1}{2}\int_{\TT^2}f(x)dx .
\end{equation}
\end{enumerate} 
\end{thm}

\begin{remark}
The statement of Theorem \ref{eqdthm} also holds for the test function 
\begin{equation}
f(x_1,x_2)=\chi_\scrD(x_1) g(x_2),
\end{equation}
where $\chi_\scrD$ is the characteristic function of a subset $\scrD\subset\TT$ with boundary of measure zero and $g\in\C(\TT)$. This follows from a standard approximation argument. 
\end{remark}

\begin{remark}
The proof of Theorem \ref{eqdthm} exploits the classic Weil bounds on (twisted) Kloosterman and Sali\'e sums. These directly yield explicit bounds on the rate of convergence in Theorem \ref{eqdthm} for smooth test functions $f$.  
\end{remark}

The following two lemmas will be helpful in proving Theorem \ref{eqdthm}.

\begin{lem}\label{lem7}Assume $q\in\NN$ is not a square. Then  there exists $r\in \ZZ$ such that $r\equiv1 \bmod 4$ and $(\frac{q}{r})=-1$.
\end{lem}

\begin{proof}
Let $q=a_0^{t_0}a_1^{t_1}...a_s^{t_s}$ be the prime factorization of $q$, with $a_0=2$. Since $q$ is not a square, at least one of the $t_i$ is odd. Suppose now this happens for index $j$, i.e., $t_j$ is odd, for some $0\leq j \leq s$. 

Fix an integer $c_0$ such that, if $j=0$ then $c_0\equiv 5\bmod 8$ and otherwise $c_0\equiv 1\bmod 8$.
Then $\big(\frac{2}{c_0}\big)=-1$ if $j=0$ and otherwise $\big(\frac{2}{c_0}\big)=1$. In both cases of course $c_0\equiv 1\bmod 4$.

Furthermore, let $c_1,\ldots,c_s$ be integers so that for $i\neq j$, $c_i$ is a quadratic residue mod $a_i$, and for $i=j$, it is not. Hence $\big(\frac{c_i}{a_i}\big)=-1$ if $i=j$ and $\big(\frac{c_i}{a_i}\big)=1$ otherwise.

By the Chinese Remainder Theorem there is an integer $r\equiv 1\bmod 4$ such that $\big(\frac{2}{r}\big)=\big(\frac{2}{c_0}\big)$ and $\big(\frac{r}{a_i}\big)=\big(\frac{c_i}{a_i}\big)$. Using quadratic reciprocity, we have for $i\neq 0$
\begin{equation}
\bigg(\frac{r}{a_i}\bigg) = (-1)^{\frac{r-1}{2}\frac{a_i-1}{2}} \bigg(\frac{a_i}{r}\bigg) =
\bigg(\frac{a_i}{r}\bigg) ,
\end{equation}
since $r\equiv 1\bmod 4$. Using the multiplicativity of the Jacobi symbol we obtain 
\begin{equation}
\bigg(\frac{q}{r}\bigg) = \bigg(\frac{2}{r}\bigg)^{t_0}\bigg(\frac{a_1}{r}\bigg)^{t_1}\cdots \bigg(\frac{a_s}{r}\bigg)^{t_s} = \bigg(\frac{a_j}{r}\bigg)^{t_j} =-1.
\end{equation}
\end{proof}

\begin{lem}\label{lem6}
\begin{enumerate}
\item[(i)] If $q\equiv 0\bmod 4$ is not a square and $\sigma\in\{\pm 1,\pm\i\}$, then
$$
 \sum_{\substack{p\in\ZZ_q^\times \\ \epsilon_p (\frac{q}{p})=\sigma}}1=\frac{1}{4}\phi(q).
$$
\item[(ii)] If $q\equiv 0\bmod 4$ and $\sigma\in\{\pm 1\}$, then
$$
 \sum_{\substack{p\in\ZZ_q^\times \\ p\equiv \sigma\bmod 4}}1=\frac{1}{2}\phi(q).
$$
\item[(iii)] If $q\equiv 1\bmod 2$ is not a square and $\sigma\in\{\pm 1\}$, then
$$
 \sum_{\substack{p\in\ZZ_q^\times \\ (\frac{p}{q})=\sigma}}1=\frac{1}{2}\phi(q).
$$
\end{enumerate}
\end{lem}

\begin{proof}
(i) We have
\begin{equation}
\sum_{\substack{p\in\ZZ_q^\times \\ \epsilon_p (\frac{q}{p})=\sigma}}1
=\frac{1}{4}\sum_{k\in\ZZ_4}   \sigma^{-k}  \sum_{p\in\ZZ_q^\times} \epsilon_p^k \bigg(\frac{q}{p}\bigg)^k.
\end{equation}
For $k=0$, we have of course,
\begin{equation}
\sum_{p\in\ZZ_q^\times}1=\phi(q) .
\end{equation}
In the case $k=1$,
we need to show that 
\begin{equation} \sum_{p\in\ZZ_q^\times} \epsilon_p \bigg(\frac{q}{p}\bigg)=0.
\end{equation}
By Lemma \ref{lem7}, we can find an $r$ such that $(\frac{q}{r})=-1$ and $r\equiv1\bmod 4$. Since  $(\frac{q}{r})=-1$ we have $r\in\ZZ_q^\times$ and thus for every $p\in\ZZ_q^\times$ there is $\tilde p\in\ZZ_q^\times$ such that $r\tilde{p}=p$. Therefore,
\begin{equation}\begin{split}\sum_{p\in\ZZ_q^\times} \epsilon_p \bigg(\frac{q}{p}\bigg)&=\sum_{\tilde p\in\ZZ_q^\times} \epsilon_{r\tilde p}\bigg(\frac{q}{r}\bigg)\bigg(\frac{q}{\tilde p}\bigg)\\
&=-\sum_{\tilde p\in\ZZ_q^\times} \epsilon_{r\tilde p}\bigg(\frac{q}{\tilde p}\bigg).
\end{split}
\end{equation}  
Since $r\equiv1\bmod4$, we have $\epsilon_{r\tilde p}=\epsilon_{\tilde p}$. Hence,
\begin{equation} \sum_{p\in\ZZ_q^\times} \epsilon_p \bigg(\frac{q}{p} \bigg)=-\sum_{\tilde p\in\ZZ_q^\times} \epsilon_{\tilde p} \bigg(\frac{q}{\tilde p} \bigg)=0.
\end{equation}
The case $k=-1$ is the complex conjugate of the case $k=1$.
For $k=2$,
\begin{equation}
\sum_{p\in\ZZ_q^\times}\epsilon_p^{2}\bigg(\frac{q}{p}\bigg)^2=
\sum_{p\in\ZZ_q^\times}\epsilon_p^{2} =0,
\end{equation}
since $\epsilon_p^2=1$ or $-1$ if $p\equiv 1$ or $3\bmod4$, respectively.

(ii) The proof follows from the $k=2$ part of the proof for (i), since the assumption that $q$ is not a square is not relevant in this case.

(iii) We have for $\sigma=\pm 1$,
\begin{equation}
\sum_{\substack{p\in\ZZ_q^\times \\ (\frac{p}{q})=\sigma}}1
=
\frac{1}{2}\sum_{k\in\ZZ_2}   \sigma^{-k}  \sum_{p\in\ZZ_q^\times}  \bigg(\frac{p}{q}\bigg)^k .
\end{equation}
For $k=0$, the sum is obviously equal to $\phi(q)$. The case $k=1$ corresponds to the well known identity
\begin{equation}\sum_{p\in\ZZ_q^\times}\bigg(\frac{p}{q}\bigg)=0 .
\end{equation}
\end{proof}

\begin{proof}[{\bf Proof of Theorem \ref{eqdthm}}]
We start with the most difficult case (ii). To prove this claim, it suffices (in view of Lemma \ref{lem6} (ii) and Weyl's criterion) to show that 
\begin{equation}
\lim_{q\rightarrow\infty}\frac{1}{\phi(q)}\sum_{\substack{p\in\ZZ_q^\times \\ \epsilon_p (\frac{q}{p})=\sigma}} e\bigg(\frac{mp+nt\overline p}{q} \bigg)=0
\end{equation}
for every fixed $(m,n)\in\ZZ^2\setminus\{(0,0)\}$, uniformly in $t\in\ZZ_q^\times$.
We have
\begin{equation}
\sum_{\substack{p\in\ZZ_q^\times \\ \epsilon_p (\frac{q}{p})=\sigma}} e\bigg(\frac{mp+nt\overline p}{q} \bigg)
=\sum_{k\in\ZZ_4}   \sigma^{-k}  \sum_{p\in\ZZ_q^\times} \epsilon_p^k \bigg(\frac{q}{p}\bigg)^k e\bigg(\frac{mp+nt\overline p}{q} \bigg) .
\end{equation}
For $k=0$ the inner sum is the Kloosterman sum
\begin{equation}
K(m,nt,q)=\sum_{p\in\ZZ_q^\times} e\bigg(\frac{mp+nt\overline p}{q} \bigg),
\end{equation}
for which we have the classical Weil bound $|K(m,nt,q)|\leq \gcd(m,nt,q)^{1/2}  q^{1/2} \tau(q)$, see \cite{Esterman}. Since $m$ and $n$ are fixed and $\gcd(t,q)=1$, $\gcd(m,nt,q)^{1/2}=\gcd(m,n,q)^{1/2}$ is bounded above. Furthermore $\tau(q)\ll_\epsilon q^\epsilon$ for any $\epsilon>0$. Since $\phi(q)\gg_\epsilon q^{1-\epsilon}$ for any $\epsilon>0$, we see that $\phi(q)^{-1}K(m,nt,q)$ tends to zero, uniformly in $t\in\ZZ_q^\times$ as $q\to\infty$, as required.

The case $k=1$ ($k=-1$) leads to the twisted Kloosterman sum
\begin{equation}\label{Stheta}
S_\theta(m,nt,q)=\sum_{p\in\ZZ_q^\times} \epsilon_p \bigg(\frac{q}{p}\bigg) e\bigg(\frac{mp+nt\overline p}{q} \bigg)
\end{equation}
(and the complex conjugate of $S_\theta(-m,-nt,q)$). Here we have the same bound as for Kloosterman sums, $|S_\theta(m,nt,q)|\leq \gcd(m,n,q)^{1/2}  q^{1/2} \tau(q)$, see \cite{Chinen} (but also the more recent \cite{Duke}). We conclude that the contribution of the $k=\pm 1$ term also tends to zero uniformly in $t$.

The case $k=2$ leads to 
\begin{equation}
\sum_{p\in\ZZ_q^\times} \epsilon_p^2 e\bigg(\frac{mp+nt\overline p}{q} \bigg)
=\sum_{p\in\ZZ_q^\times}  e\bigg(\frac{p-1}{4}+\frac{mp+nt\overline p}{q} \bigg) 
= -\i K\bigg(m+\frac{q}{4},nt,q\bigg)
\end{equation}
and is thus reduced to Kloosterman sums.
This proves the case (ii). 

Case (iii) reduces to the same estimates as in case (ii) $k=2$.

Case (iv) is analogous, but here the estimates reduce to bounds on Sali\'e sums
\begin{equation}\label{Salie}
S(m,nt,q)=\sum_{p\in\ZZ_q^\times} \bigg(\frac{p}{q}\bigg) e\bigg(\frac{mp+nt\overline p}{q} \bigg),
\end{equation}
which are the same as the above for the (twisted) Kloosterman sums.

Case (i) of course follows from the classical Weil bound for Kloosterman sums.
\end{proof}

\section{Proof of Theorem \ref{thm1}\label{proof}}

{\bf Case 1a:} $q\equiv 0 \bmod 4$, $q$ not a square. We need to show that for any bounded continuous $F:\CC\to\CC$ we have
\begin{equation}\label{say}
\frac{1}{\phi(q)}\sum_{\substack{p\in\ZZ_q^\times \\ \epsilon_p (\frac{q}{p})=\sigma}}
\chi_\scrD\bigg(\frac{p}{q}\bigg) F\bigg( \frac{g_\varphi(p,q)}{g_1(p,q)} \bigg)\rightarrow\frac{|\scrD|}{4}\int_{\TT}F(G_\varphi^+(x))dx .
\end{equation}
In view of Theorem \ref{FEthm}, this is equivalent to 
\begin{equation}
\frac{1}{\phi(q)}\sum_{\substack{p\in\ZZ_q^\times \\ \epsilon_p (\frac{q}{p})=\sigma}}
\chi_\scrD\bigg(\frac{p}{q}\bigg) F\bigg( G_\varphi^+\bigg(-\frac{\overline p}{q}\bigg) \bigg)\rightarrow\frac{|\scrD|}{4}\int_{\TT}F(G_\varphi^+(x))dx .
\end{equation}
Since $G_\varphi^+$ and $F$ are continuous, the latter statement follows from Theorem \ref{eqdthm} (ii) and subsequent remark, if we choose the test function
\begin{equation}
f(x_1,x_2) = \chi_\scrD(x_1) F(G_\varphi^+(-x_2)) .
\end{equation}

{\bf Case 1b:} $q\equiv 0 \bmod 4$, $q$ is a square. 
We proceed as in Case 1b, and note that the condition $\epsilon_p=1$ ($\epsilon_p=\i$) is equivalent to $p\equiv 1\bmod 4$ ($p\equiv -1\bmod 4$). The statement follows from Theorem \ref{eqdthm} (iii).

{\bf Case 2a:} $q\equiv 1 \bmod 2$, $q$ not a square. In this case, the statement to be proved reduces (again using Theorem \ref{FEthm}) to 
\begin{equation}
\frac{1}{\phi(q)}\sum_{\substack{p\in\ZZ_q^\times \\ (\frac{p}{q})=\sigma}}
\chi_\scrD\bigg(\frac{p}{q}\bigg) F\bigg( G_\varphi\bigg(-\frac{\overline{4p}}{q}\bigg) \bigg)\rightarrow\frac{|\scrD|}{2}\int_{\TT}F(G_\varphi(x))dx .
\end{equation}
which follows from Theorem \ref{eqdthm} (iv) with $t=\overline 4$. 

{\bf Case 2b:} $q\equiv 1 \bmod 2$, $q$ is a square. Analogous to Case 2a, except that we employ Theorem \ref{eqdthm} (i).

{\bf Case 3a:} $q\equiv 2 \bmod 4$, $q$ not a square. Following the same strategy as above, we deduce that the claim of the theorem is equivalent to
\begin{equation}\label{prev}
\frac{1}{\phi(q)}\sum_{\substack{p\in\ZZ_q^\times \\ (\frac{2p}{q/2})=\sigma}}
\chi_\scrD\bigg(\frac{p}{q}\bigg) F\bigg( G_\varphi^-\bigg(-\frac{\overline{8p}}{q/2}\bigg) \bigg)\rightarrow\frac{|\scrD|}{2}\int_{\TT}F(G_\varphi^-(x))dx .
\end{equation}
As in the proof of Theorem \ref{FEthm} (iii), we substitute $q=2q_0$ and $p=2p_0+q_0$, i.e., $q_0=q/2$ and $p_0=\frac14(2p-q)$. Note that this map describes a bijection $\ZZ_q^\times\to \ZZ_{q_0}^\times$. Hence \eqref{prev} is equivalent to
\begin{equation}
\frac{1}{\phi(q)}\sum_{\substack{p\in\ZZ_{q_0}^\times \\  (\frac{p_0}{q_0})=\sigma}}
\chi_\scrD\bigg(\frac{p_0}{q_0}+\frac12 \bigg) F\bigg( G_\varphi^-\bigg(-\frac{\overline{16 p_0}}{q_0}\bigg) \bigg)\rightarrow\frac{|\scrD|}{2}\int_{\TT}F(G_\varphi^-(x))dx ,
\end{equation}
which is again implied by Theorem \ref{eqdthm} (iv).

{\bf Case 3b:} $q\equiv 2 \bmod 4$, $q$ is a square. Analogous to Case 3a, except that we use Theorem \ref{eqdthm} (i).
\qed

\section{Mean-square estimates\label{variasec}}

The key step in the proof of Theorem \ref{thm2} is the estimate on the mean-square given in Lemma \ref{RIlem}.

\begin{proof}[Proof of Lemma \ref{RIlem}]
We have
\begin{equation}
\begin{split}
M_{2,\varphi}(q) & \leq \frac{1}{|\scrD|\phi(q)} \sum_{m\in\ZZ_q} |g_\varphi(m,q)|^2 \\
& \leq \frac{q}{|\scrD|\phi(q)} \sum_{\substack{h,h'\in\ZZ_q\\ h^2\equiv {h'}^2\bmod q}} \bigg| \varphi\bigg(\frac{h}{q}\bigg) \varphi\bigg(\frac{h'}{q}\bigg) \bigg| .
\end{split}
\end{equation}
Take $\psi\in\scrB(\TT)$ such that $|\varphi(x)|\leq\psi(x)$ for all $x\in\TT$. Then
\begin{equation}\label{shsm}
\begin{split}
M_{2,\varphi}(q) & \leq \frac{q}{|\scrD|\phi(q)} \sum_{\substack{h,h'\in\ZZ_q\\ h^2\equiv {h'}^2\bmod q}} \psi\bigg(\frac{h}{q}\bigg) \psi\bigg(\frac{h'}{q}\bigg) \\
& = \frac{1}{|\scrD|\phi(q)} \sum_{m\in\ZZ_q} |g_\psi(m,q)|^2 \\
& = \frac{1}{|\scrD|\phi(q)} \sum_{r|q} \sum_{p\in\ZZ_{q/r}^\times} |g_{\psi_r}(p,q/r)|^2 
\end{split}
\end{equation}
where (recall \eqref{reduct})
\begin{equation}
\psi_r(x)=\sum_{k=0}^{r-1} \psi\bigg( \frac{x+k}{r} \bigg) .
\end{equation}
The Fourier series of this function is
\begin{equation}
\psi_r(x)=r \sum_{n\in\ZZ} \widehat\psi_{rn} e(nx) ,
\end{equation}
where $\widehat\psi_k$ are the Fourier coefficients of $\psi$.
We have thus shown that
\begin{equation}
M_{2,\varphi}(q) \leq \frac{1}{|\scrD|\phi(q)} \sum_{r|q} \phi\bigg(\frac{q}{r}\bigg) M_{2,\psi_r}^{(\scrD=\TT)}\bigg(\frac{q}{r}\bigg)
\leq \frac{q}{|\scrD|\phi(q)} \sum_{r|q} \frac1r \;M_{2,\psi_r}^{(\scrD=\TT)}\bigg(\frac{q}{r}\bigg) .
\end{equation}
By Corollary \ref{cor2}, for every fixed $r$,
\begin{equation}
\lim_{q\to\infty} \frac{1}{q} M_{2,\psi_r}\bigg(\frac{q}{r}\bigg) \leq 
2r\bigg(|\widehat\psi_0|^2 + \sum_{n=1}^\infty  |\widehat\psi_{rn}+\widehat\psi_{-rn}|^2\bigg)
\leq 
2r\bigg(|\widehat\psi_0|^2 + \sum_{n=1}^\infty  |\widehat\psi_{n}+\widehat\psi_{-n}|^2 \bigg).
\end{equation}
The right hand side is bounded by $4r\| \psi \|_2^2$.
The convergence is uniform in $r$, if we assume $\psi$ has a finite Fourier series. In this case,
we therefore have
\begin{equation}\label{varia22}
\limsup_{\substack{q\to\infty \\ d(q)\leq N}} \frac{M_{2,\varphi}(q)}{q} \leq \frac{C_N}{|\scrD|} \| \psi \|_2^2.
\end{equation}

Since $\varphi$ is Riemann integrable, given any $\epsilon>0$, there exist $\psi$ with finite Fourier series such that (a) $|\varphi|<\psi$ (as required in \eqref{shsm}) and (b) $\| |\varphi|-\psi \|_2 <\epsilon$. This proves that the right hand side of \eqref{varia22} is arbitrarily close to $\frac{C_N}{|\scrD|}\,\| \varphi\|_2^2$. 
\end{proof}

\section{Proof of Theorem \ref{thm2}\label{extens}}

The following lemma says that the sequence probability measures defined by the value distribution of incomplete Gauss sums is tight. By the Helly-Prokhorov theorem, this means that the sequence is relatively compact, i.e., every sequence contains a convergent subsequence. 

\begin{lem}\label{lemA}
Fix $N$. For every $\epsilon>0$ there exists $K_\epsilon>0$ such that
\begin{equation}
\limsup_{\substack{q\to\infty \\ d(q)\leq N}}  \frac{1}{\phi(q)} \big|\{ p\in\ZZ_q^\times : q^{-1/2} |g_\varphi(p,q)|> K_\epsilon \}\big| < \epsilon \|\varphi\|_2^2
\end{equation}
for any Riemann integrable $\varphi:\TT\to\CC$. 
\end{lem}

\begin{proof}
We have, by Chebyshev's inequality   
\begin{equation}\label{cheby}
\frac{1}{\phi(q)} \big|\{ p\in\ZZ_q^\times : q^{-1/2} |g_\varphi(p,q)|> K \}\big| < 
\frac{ M_{2,\varphi}(q)}{K^2 q}.
\end{equation}
The claim now follows from Lemma \ref{RIlem}.
\end{proof}

Chebyshev's inequality \eqref{cheby} also implies the following.

\begin{lem}\label{lemB}
Fix $N$. Let $\varphi:\TT\to\CC$ be Riemann integrable. Then, for every $\epsilon>0$, $\delta>0$ there exists $\psi\in\scrB(\TT)$ and such that
\begin{equation}
\limsup_{\substack{q\to\infty \\ d(q)\leq N}}  \frac{1}{\phi(q)} \big|\{ p\in\ZZ_q^\times : q^{-1/2} |g_\varphi(p,q)-g_\psi(p,q)|> \delta \}\big| < \epsilon .
\end{equation}
\end{lem}

\begin{proof}
This follows immediately from \eqref{cheby}; note that $g_\varphi(p,q)-g_\psi(p,q)=g_{\varphi-\psi}(p,q)$ and $\varphi-\psi$ is Riemann integrable.
\end{proof}

We now turn to the proof of Theorem \ref{thm2}. We restrict ourselves to Case 1a where $q\equiv 0\bmod 4$; the other cases are analogous. The relative compactness implied by Lemma \ref{lemA} can be stated as follows. Any sequence of $q\to\infty$ with $d(q)\leq N$ contains a subsequence $\{q_j\}$ with the property: there is a probability measure $\nu$ on $\{\pm1\pm\i\}\times\CC$ such that for any $\sigma\in\{\pm1\pm\i\}$ and any bounded continuous function $F:\CC\to\CC$ we have
\begin{equation}\label{lili}
\lim_{j\to\infty}\frac{1}{|\scrD|\phi(q_j)}\sum_{\substack{p\in\ZZ_{q_j}^\times\cap q_j\scrD \\ \epsilon_p (\frac{q_j}{p})=\sigma}} F\bigg( \frac{g_\varphi(p,q_j)}{g_1(p,q_j)} \bigg) = 
\int_{\CC}F(z) \nu(\sigma,dz) .
\end{equation}
The probability measure $\nu$ may depend on the choice of subsequence, $\varphi$ and on $\scrD$. 

Let us now show that for every $F\in\C_0^\infty(\CC)$ (infinitely differentiable and of compact support) the limit 
\begin{equation}\label{IF}
I_\varphi(F):=\lim_{\substack{q\to\infty \\ d(q)\leq N}}\frac{1}{|\scrD|\phi(q)}\sum_{\substack{p\in\ZZ_{q}^\times\cap q\scrD \\ \epsilon_p (\frac{q}{p})=\sigma}} F\bigg( \frac{g_\varphi(p,q)}{g_1(p,q)} \bigg)
\end{equation}
exists. $I_\varphi(F)$ must then be equal to the right hand side of \eqref{lili}, which in fact means that $\nu$ is unique and the full sequence of $q$ converges. 

To prove the existence of $I_\varphi(F)$, note first of all that since $F\in\C_0^\infty(\CC)$ we have $|F(w)-F(z)| \leq C \min\{1,|w-z|\}$ for some constant $C>0$. Therefore, for $\psi$, $\delta$, $\epsilon$ as in Lemma \ref{lemB}, we have
\begin{equation}
\begin{split}\label{ineq000}
& \frac{1}{|\scrD|\phi(q)}\sum_{\substack{p\in\ZZ_{q}^\times\cap q\scrD \\ \epsilon_p (\frac{q}{p})=\sigma}} \bigg| F\bigg( \frac{g_\varphi(p,q)}{g_1(p,q)} \bigg)-
F\bigg( \frac{g_\psi(p,q)}{g_1(p,q)} \bigg) \bigg| \\
& \leq \frac{C}{|\scrD|\phi(q)}\sum_{\substack{p\in\ZZ_{q}^\times\cap q\scrD \\ \epsilon_p (\frac{q}{p})=\sigma}} \min\bigg\{ 1, \bigg| \frac{g_\varphi(p,q)}{g_1(p,q)} -\frac{g_\psi(p,q)}{g_1(p,q)} \bigg|\bigg\} \\
& \leq \frac{C}{|\scrD|\phi(q)}\sum_{p\in\ZZ_{q}^\times} \min\bigg\{ 1, \bigg| \frac{g_\varphi(p,q)}{g_1(p,q)} -\frac{g_\psi(p,q)}{g_1(p,q)} \bigg|\bigg\} \\
& \leq \frac{C}{|\scrD|} ( 2^{1/2} \delta +\epsilon ) .
\end{split}
\end{equation}
Since the limit $I_\psi(F)$ exists by Theorem \ref{thm1}, the sequence 
\begin{equation}
\frac{1}{|\scrD|\phi(q)}\sum_{\substack{p\in\ZZ_{q}^\times\cap q\scrD \\ \epsilon_p (\frac{q}{p})=\sigma}} F\bigg( \frac{g_\psi(p,q)}{g_1(p,q)} \bigg)
\end{equation}
defines a Cauchy sequence. Using this fact, the bound \eqref{ineq000} and the triangle inequality, we see that 
\begin{equation}\label{say2}
\frac{1}{|\scrD|\phi(q)}\sum_{\substack{p\in\ZZ_{q}^\times\cap q\scrD \\ \epsilon_p (\frac{q}{p})=\sigma}} F\bigg( \frac{g_\varphi(p,q)}{g_1(p,q)} \bigg)
\end{equation}
is a Cauchy sequence, too, and hence $I_\varphi(F)$ exists. As mentioned earlier, this means that $I_\varphi(F)$ must then be equal to the right hand side of \eqref{lili}, which in fact means that $\nu$ is unique and the full sequence of $q$ converges for every bounded continuous $F$.

The bound \eqref{ineq000} furthermore implies that $I_\psi(F)\to I_\varphi(F)$, as $\psi\to\varphi$ in $\L^2(\TT)$. This completes the proof of Theorem \ref{thm2}.

\section{Numerics\label{secNumerics}}

The computations used in Figures \ref{fig5012}--\ref{fig5014} were carried out with Mathematica. We encoded the real and imaginary part of the incomplete Gauss sum $\frac{g_\varphi(p,q)}{g_1(p,q)}-\frac{T}{q}$ (where $\varphi$ is the characteristic function of the interval $[0,\frac{T}{q}]\subset[0,1]$) as
\begin{verbatim}
ReGauss[p_, q_, T_] := 
 If[GCD[p, q] == 1, 
  Re[Sum[Exp[2*Pi*I*h^2*p/q], {h, 1, T}]/
   Sum[Exp[2*Pi*I*h^2*p/q], {h, 1, q}]] - T/q, Infinity]
ImGauss[p_, q_, T_] := 
 If[GCD[p, q] == 1, 
  Im[Sum[Exp[2*Pi*I*h^2*p/q], {h, 1, T}]/
   Sum[Exp[2*Pi*I*h^2*p/q], {h, 1, q}]], Infinity]
\end{verbatim}
and formed a table comprising the values for all integers $p<q$. Whenever $\gcd(p,q)\neq 1$ the value $\infty$ is assigned, which is ignored by Mathematica's {\tt Histogram} command. 

The probability density of real/imaginary part of $G_\varphi^+$ and $G_\varphi$ in Figures \ref{fig5012} and \ref{fig5013} was plotted via the {\tt SmoothHistogram} command, where we truncated the Fourier series $G_\varphi^+(x)$ and $G_\varphi(x)$ at $n=4000$ and sampled $x$ at 300,000 random points in $[0,1]$. As the distribution of real and imaginary part of $G_\varphi^-$ are the same, we only computed $\Im G_\varphi^-$ in Figure \ref{fig5014}, truncated at $n=5000$ and with 500,000 sample points.

\end{document}